\documentclass{amsart}

\usepackage{setspace}
\usepackage{graphicx}
\usepackage{verbatim}   
\usepackage{color}      
\usepackage{subfigure}  
\usepackage{hyperref}   

\newtheorem{thm}{Theorem}[section]

\newtheorem{lemma}[thm]{Lemma}

\theoremstyle{definition}

\numberwithin{equation}{section}
\def\BMT{\lower0.9em\hbox{\includegraphics{BMT.pdf}}}
\def\TMB{\lower0.9em\hbox{\includegraphics{TMB.pdf}}}
\def\Msplit{\lower0.8em\hbox{\includegraphics{Msplit.pdf}}}
\def\Tsplit{\lower1.0em\hbox{\includegraphics{Tsplit.pdf}}}
\def\Bsplit{\lower1.0em\hbox{\includegraphics{Bsplit.pdf}}}
\def\NewTsplit{\lower1.0em\hbox{\includegraphics{Tsplit2.pdf}}}
\def\NewBsplit{\lower1.0em\hbox{\includegraphics{Bsplit2.pdf}}}
\def\Ufirstsplit{\lower0.7em\hbox{\includegraphics{Ufirstsplit.pdf}}}
\def\Usecondsplit{\lower0.7em\hbox{\includegraphics{Usecondsplit.pdf}}}

\begin{document}
\title[Generalized Augmented Alternating Links]
{Generalized Augmented Alternating Links \\ and Hyperbolic Volumes}

\date{\today}
\author[Colin Adams]{Colin Adams}
\address{Department of Mathematics and Statistics, Williams College, Williamstown, MA 01267}
\email{Colin.C.Adams@williams.edu}

\begin{abstract} Augmented alternating links are links obtained by adding trivial components that bound twice-punctured disks to non-split reduced non-2-braid prime alternating projections. These links are known to be hyperbolic. Here, we extend to show that generalized augmented alternating links, which allow for new trivial components that bound $n$-punctured disks, are also hyperbolic. As an application we consider generalized belted sums of links and compute their volumes.
\end{abstract}
\maketitle

\section{Introduction}\label{S:intro}

A non-split link in $S^3$ is known to be either hyperbolic or to contain an essential torus or annulus in its complement. When the link is hyperbolic, its complement admits a hyperbolic metric that is uniquely determined and hence, the hyperbolic volume of its complement becomes an invariant that can be used to distinguish it from other links. 

In  \cite{Men}, Menasco proved that prime alternating non 2-braid links are hyperbolic. In \cite{Adams1}, it was further proven that augmented alternating links  are hyperbolic. These links are obtained from a prime non 2-braid alternating link projection by adding trivial ``vertical" components perpendicular to the projection plane  that bound a disk punctured twice by the alternating link. These augmented alternating links have proved useful in a variety of contexts. In particular,  they appear as the geometric limits of alternating links that correspond to twisting the two strands around which the augmenting components wrap. As such, together with the alternating links, they form the closure of the collection of alternating links in the geometric topology (see \cite{Lack}). In particular, the volumes of the links in such a sequence must approach the volume of  the augmented link from below.

If a link is not alternating, one can augment it at a subset of the crossings that would need to be changed to make it alternating, and then the result has the same complement as an augmented alternating link, since a full twist along one of the twice-punctured disks can reverse that crossing. 

Augmented alternating links are useful in obtaining upper bounds on volumes of links. See for instance \cite{Lack} and \cite{Pur1}. In some papers, the links considered are maximally augmented. That is to say, every crossing in the original knot is in a twist sequence around which a vertical component has been added. In that case, one can use Andreev's Theorem to prove hyperbolicity (c.f. \cite{Pur6}). 
       
In this paper, we extend the results of \cite{Adams1} to allow the vertical components to bound disks that are punctured more than twice by the alternating link in the projection plane. These new links are called generalized augmented alternating links. A precise definition appears in Section 2. Our main theorem is to prove that indeed, their complements are always hyperbolic.

In particular, since $(1,q)$-Dehn filling of these vertical components corresponds to adding $q$ full twists to the strands of the original link, the fact the augmented link is hyperbolic implies that the resulting links are always hyperbolic for high enough values of the $q's$. 

Note that in several papers, authors have considered generalized augmented links that were also obtained by adding vertical components to a projection, but in this case, not necessarily alternating, such that the projection breaks up into generalized twist regions as in Figure \ref{twist}. But here again,  each twist region has to receive a crossing circle. These links have a variety of interesting properties as discussed in \cite{FKP}, \cite{Pur2}, \cite{Pur5}, \cite{Pur3} and \cite{Pur4}.

\begin{figure}[h]
\begin{center}
\includegraphics[scale=0.7]{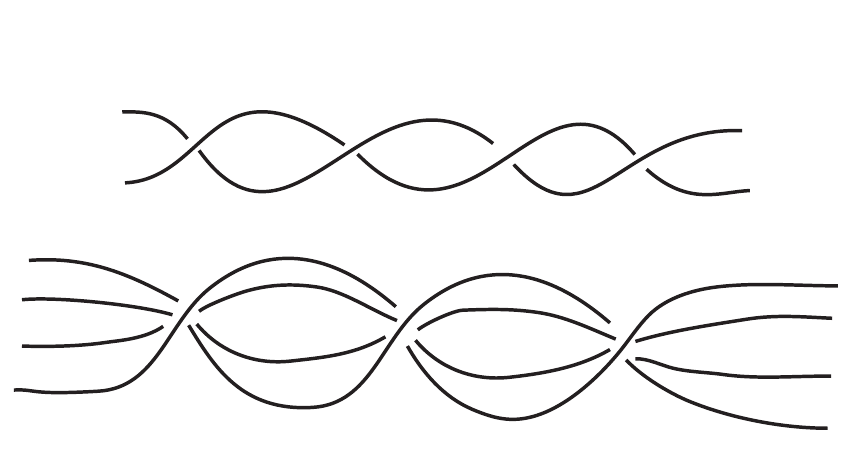}
\caption{A traditional twist region and a generalized twist region.}
\label{twist}
\end{center}
\end{figure}

The presence of  twice-punctured disks  in link complements and the fact twice punctured disks are totally geodesic with a unique hyperbolic structure (see \cite{Adams2}) implies that one can take belted sums of the links. The resulting link $L_1 \#_b L_2$ has volume equal to the sum of the volumes of the two links $L_1$ and $L_2$ that are summed, as in Figure \ref{beltedsum}(a).

In Section 3, we generalize the notion of belted sum, and find an explicit  formula  between the volumes  of the two links and their summand. Specifically, we show that if a hyperbolic link denoted $L'_1 \#_b L'_2$ is constructed from two links $L'_1$ and $L'_2$ as in Figure \ref{beltedsum}(b), then $\mbox{vol}(L'_1 \#_bL'_2) = \mbox{vol}(L'_1) + \mbox{vol}(L'_2) - 4(3.6638...)$. Similarly, if $L''_1 \#_bL''_2$ is a link constructed as in Figure \ref{beltedsum}(c), then $\mbox{vol}(L''_1 \#_bL''_2) = \mbox{vol}(L''_1) + \mbox{vol}(L''_2) - 8(3.6638...)$. In this case, there are two distinct options for the central belt, wrapping either lower left to upper right or lower right to upper left. 

\begin{figure}[h]
\begin{center}
\includegraphics[scale=0.7]{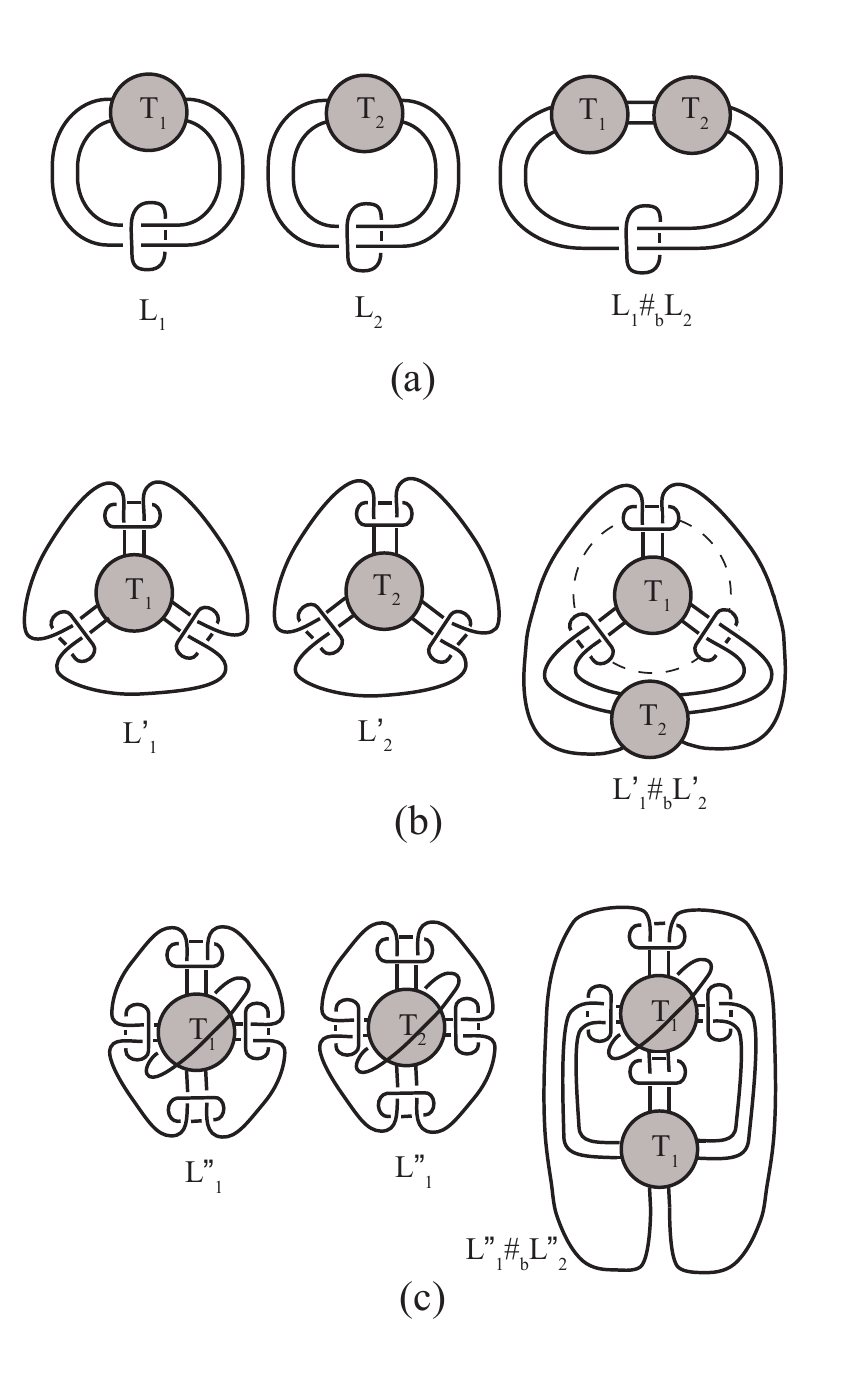}
\caption{Generalized belted sums of links.}
\label{beltedsum}
\end{center}
\end{figure}

More generally, let $L_1$ and $L_2$ be two links, each with  a $2n$-string tangle at center, with $n$ belts around adjacent pairs of the exiting strings and $n-3$ belts around the central tangle in the same pattern, as for instance appears in Figure \ref{beltedsum2} in the case $n=5$.  Then $\mbox{vol}(L_1 \#_bL_2) = \mbox{vol}(L_1) +\mbox{vol}l(L_2) - 4(n-2)3.6638...$. This construction answers a question asked by Oliver Dasbach about the behavior of the volumes in Figure 1(b), and was motivated by that question.

\begin{figure}[h]
\begin{center}
\includegraphics[scale=0.7]{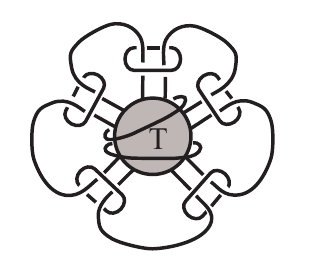}
\caption{A potential generalized belted sum factor link.}
\label{beltedsum2}
\end{center}
\end{figure}

Of course, given a specific link, we would like to add vertical components to obtain either a link that can be decomposed via belted sum into simpler links or composed via belted sum into more complicated links. But we need to know that the resulting link is hyperbolic. This is what the main theorem provides when the initial link is alternating. 

In the case of augmented alternating links, the fact they are hyperbolic and the fact that twice-punctured disks are totally geodesic with a unique hyperbolic structure implies that twisting a half-twist on the twice-punctured disks bounded by a vertical component, which adds or subtracts a crossing,  will yield a new link complement that is hyperbolic with the same volume as the original link. However, in the case of generalized augmented alternating links, if we twist a half-twist on an $n$-punctured disk bounded by a vertical component for $n \ge 3$, the result need not be hyperbolic, and even if it is, the volume is generally not preserved. As an example, adding one vertical component bounding a thrice-punctured disk in the figure-eight knot complement, and then twisting a half-twist yields a Seifert fibered space. Further applications of generalized augmented alternating links  to volume bounds for links appear in \cite{Adams3}. 

\section{Hyperbolicity}

Given an alternating link $J$ in a reduced alternating projection $P$, we will often consider it a 4-regular graph on the projection sphere. That graph cuts the sphere up into complementary regions. 

 Let $J$ be a prime non-split non-2-braid alternating link. Let $P$ be a reduced alternating projection of $J$. Note that by results of Menasco in \cite{Men}, the projection is connected and there are no simple closed curves in the plane that intersect the projection transversely twice such that there are crossings to either side of the curve. Choose two complementary regions in the projection plane that do not share an edge. Take a trivial component $C$ that intersects the projection sphere in precisely one point in each of the complementary regions. Then we say that $J' = J \cup C$ is a {\it generalized singly augmented alternating link}.  We call the additional component a {\it vertical component}.

In the projection plane, we keep track of the vertical component as a grey arc $\gamma$. While fixing endpoints, we can isotope $\gamma$ to minimize the number of intersections with the link $J$. This corresponds to an isotopy of the vertical component in the complement of $J$. We will assume that such an isotopy has already taken place, and call the corresponding vertical component a minimal representative of the isotopy class. 

For any other pair of non-adjacent complementary regions, we allow the introduction of additional vertical components, as long as there are minimal representations of all the individual vertical components that are disjoint as arcs in the plane. We call the resulting link a {\it generalized augmented alternating link}. Note that for any pair of non-adjacent regions there can be at most one corresponding vertical component. There can be quite a few vertical components as for instance appears in the fully augmented figure-eight knot in Figure \ref{fullyaugmented}. Note that for the figure-eight knot, there is more than one option for a fully augmented link that results.

\begin{figure}[h]
\begin{center}
\includegraphics[scale=0.4]{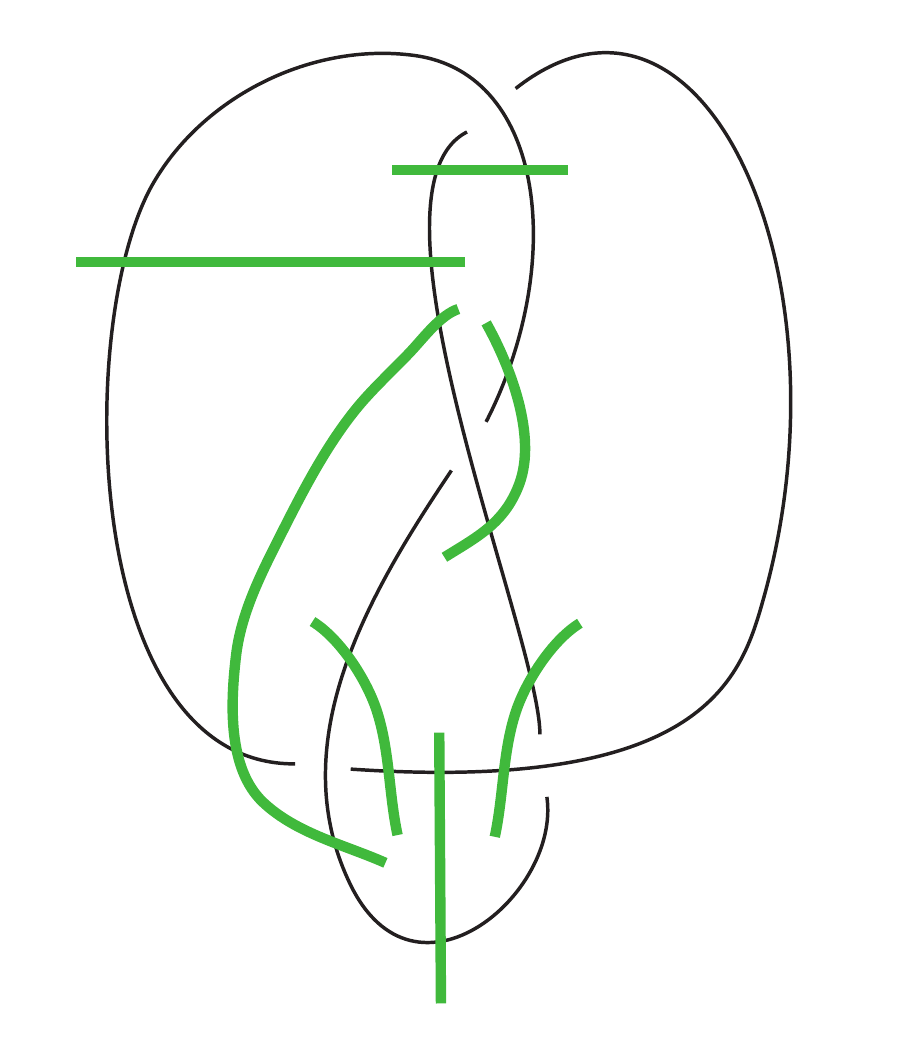}
\caption{A generalized augmented projection of the figure-eight knot with the maximum possible number of vertical components.}
\label{fullyaugmented}
\end{center}
\end{figure}

It should be noted that there are reduced alternating projections such that not all of the possible vertical components can be added since minimal representations overlap as in Figure \ref{intersection}.

\begin{figure}[h]
\begin{center}
\includegraphics[scale=0.4]{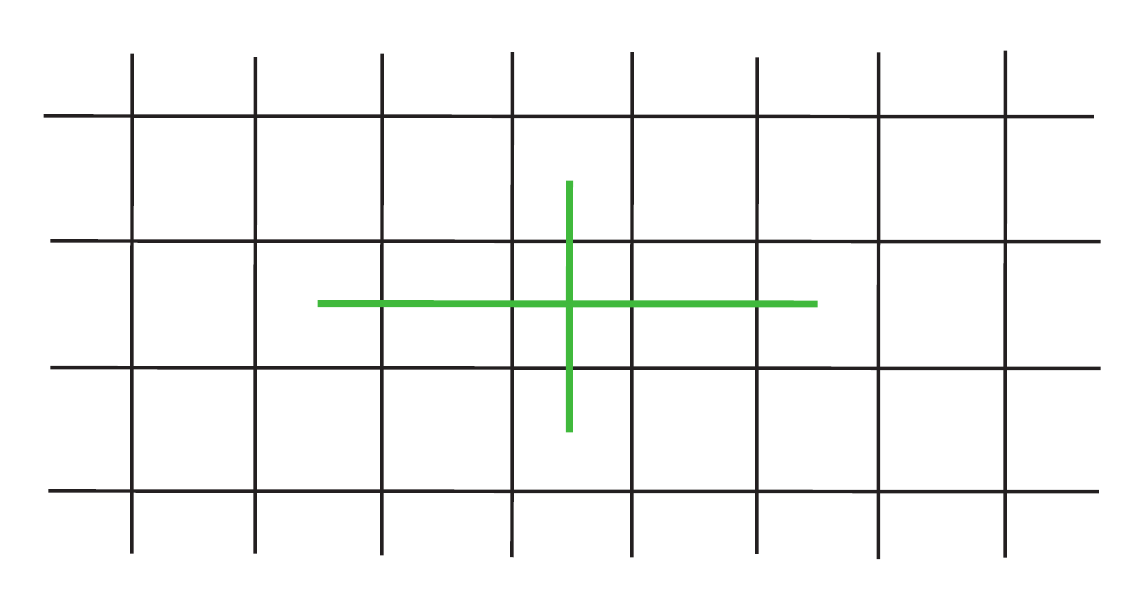}
\caption{These two vertical components on this alternating grid cannot be made to avoid intersecting while in minimal representations.}
\label{intersection}
\end{center}
\end{figure}

\begin{thm} \label{main} Let $J$ be a prime non-2-braid non-split alternating link. Then any generalized augmented alternating link $J'$ constructed from a reduced alternating projection of $J$ is hyperbolic.
\end {thm}

In fact, even in the case of a 2-braid link in a reduced alternating projection, if we add a vertical component that does not correspond to the axis around which the knot is braided, the result is hyperbolic. However, we will not include that case here.

It was proved in \cite {Men} that $J$ is hyperbolic. Here, we are proving that the addition of these new vertical component preserves hyperbolicity. We  use the machinery developed in \cite{Men} as described below.

We consider the projection plane $P$ as a subspace of the projection sphere obtained by the 1-point compactification. We will move from the plane to the sphere without comment. 

Let $L$ be a link in a projection. At each crossing of the projection, a bubble is inserted, with the overstrand going over the top of the bubble and the understrand going under the bubble. We denote $S_+$ to be the sphere obtained by replacing each equatorial disk of a bubble in the projection sphere with the top hemisphere of the corresponding bubble. We define $B_+$ to be the ball bounded by  $S_+$. Similarly, we define $S_-$ and $B_-$, using the bottom hemispheres.  In Lemma 1 of \cite{Men}, Menasco showed that for any projection, and any surface that is incompressible and not boundary parallel, the surface can be isotoped so that it intersects bubbles in saddles and each intersection curve with $S_{\pm}$ intersects each bubble at most once and intersects at least one bubble. 

 We begin with the following lemma.

\begin{lemma} \label{incom} If $C$ is a vertical component in a generalized augmented alternating link $J'$ such that $C$ is a minimal representative of its isotopy class, then the vertical punctured disk $D$ that is bounded by $C$ is incompressible.
\end{lemma}

\begin{proof} We drop all the vertical components except for $C$, since if $D$ is incompressible in $S^3-(J \cup C)$, it is incompressible in  $S^3 - J'$. Let $\gamma$ denote the arc in the projection plane that is the projection of $C$.

Suppose that $D$ compressed. Let $E$ be a compressing disk, in general position with respect to the projection of $J$. Applying the Menasco machinery to the projection of $J$ (leaving $C$ vertical), we consider how $E$ intersects the bubbles  and the spheres $S_+$ and $S_-$. Those intersections decompose $E$  into saddle disks corresponding to where it intersects bubbles, and over-disks which have interior above $S_+$ and boundary on $S_+$ and under-disks, which have interior below $S_-$ and boundary on $S_-$. We call the resultant graph on $E$, where saddles are treated as vertices, the intersection graph.

We assume that $E$ has the minimum number of saddles for a compressing disk of $D$.  Suppose first that there is a simple closed intersection curve $\alpha$ on $E$. By this we mean a simple closed curve in the intersection graph that avoids $\partial E$ and that forms a component of the boundary of  a region containing no other intersection curves on $E$. Then it is also a simple closed curve on either $S_+$ or $S_-$. For convenience, assume $S_+$.  Note that all intersection arcs that begin and end on $\partial E$ must begin and end on the same side of $\gamma$ in the projection plane. Since a simple closed intersection curve does not intersect $\gamma$, all the intersection arcs lie to the side of $\alpha$ containing $\gamma$. Take an innermost intersection curve $\alpha'$ to the other side of $\alpha$. If there are none, take $\alpha$ itself considered innermost to the outside. Since the projection is alternating, each intersection curve must intersect bubbles such that the overstrand of the bubble is alternately on the left and right of the curve. This forces $\alpha'$ to hit a bubble twice. But $\alpha'$ bounds a disk $E'$ on $E$. In the case $\alpha'$ hits both sides of a bubble, by taking an arc through $E'$ connected at both ends to an arc through the corresponding saddle, we create a loop in $E$ that encircles the link $J$ once. But that loop bounds a disk on $E$, a contradiction. In the case $\alpha'$ hits the same side of a bubble twice, we can isotope a neighborhood of an arc on $E'$ through the bubble and eliminate two saddles, again a contradiction.

Thus, any such simple closed curve must avoid all bubbles. But then we could replace the disk it bounds on $E$ with the disk it bounds on $S_+$, and push off to lower the number of intersection curves.

Hence, there are no simple closed intersection curves on $E$. All intersection curves are arcs that begin and end on $\partial E$. Note that when viewed in the plane, no such intersection arc crosses $\gamma$ but all intersection arcs  start and end on $\gamma$, all coming out one side of it. See Figure \ref{newfork} for the picture. 

As in \cite{Men}, we can similarly show that no intersection arc intersects a bubble more than once.

A {\it fork} of the intersection graph is a vertex with at least three edges ending on $\partial E$. We show that every intersection graph has at least one fork. Since every complementary region must intersect $\partial E$ in its boundary, the graph obtained by throwing away all edges with an endpoint on the boundary of $E$ is a collection of trees. Every tree of two or more vertices always has at least two leaves, and those leaves will have three edges that must all end on the boundary. So in this case, there are at least two forks. The one exception is if there is only one vertex to one tree, which coincides with the case of there being only one saddle in $E$. However, then we still have a fork.

We consider what a fork tells us about the projection $P$. See Figure \ref{newfork}.The two disks bounded by the one saddle and the three curves ending on the boundary of $E$ cause there to be exactly two arcs of the knot that come out of the crossing in question and then pass through $E$ as punctures, without crossing any other arcs of the knot in between. By Theorem 1(b) of \cite{Men}, which shows that an alternating knot is prime if and only if it is obviously so in any alternating projection, these arcs cannot contain any crossings between when they puncture $E$ and when they pass through the crossing in question.

\begin{figure}[h]
\begin{center}
\includegraphics[scale=0.7]{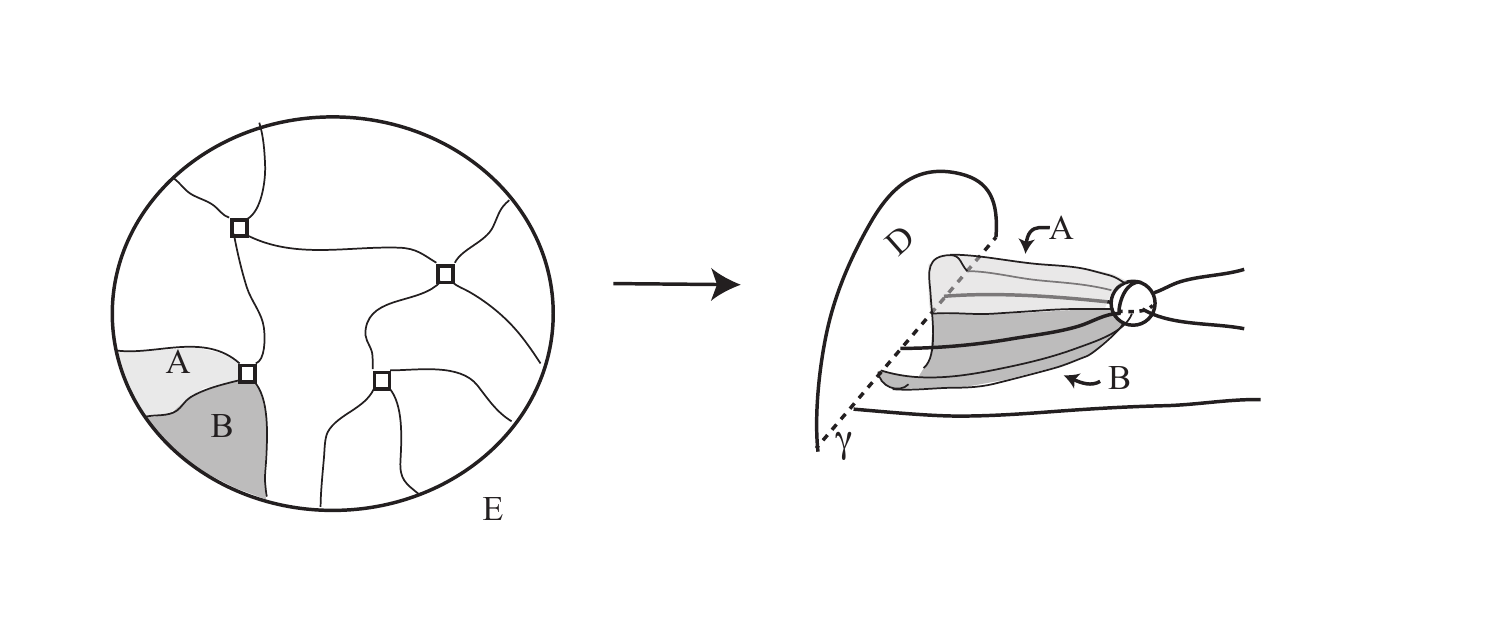}
\caption{A fork in the compression disk allows us to isotope $\gamma$ past a crossing.}
\label{newfork}
\end{center}
\end{figure}

Fixing its endpoints, we can  isotope $\gamma$ past the resultant crossing. There is a corresponding isotopy of $C$, $D$ and $E$, changing the pattern of intersections on $E$ and eliminating at least one saddle on $E$. We repeat this process until either  all saddles are eliminated from $E$, leaving only simple arcs of intersection or there is only one saddle remaining. In the first case,  taking an outermost arc  $\alpha''$ on $E$, we cut off a disk $E''$ which intersects the projection plane in an arc that does not intersect a bubble. Hence we can either isotope an arc on $\gamma$ to this arc, and lower the number of intersections in $\gamma \cap J$, a contradiction to $\gamma$ corresponding to a minimal representative of $C$, or if no part of $J$ lies in the region in the projection plane cut off by $\gamma \cup \alpha''$, we can isotope $E$ to lower the number of intersection arcs. In either case, repeating the process if necessary, we obtain a contradiction.

In the second case, if only one saddle remains, then as in Figure \ref{newfork2}, the arc $\gamma$ can be isotoped  with endpoints fixed to lower its number of punctures, and it is therefore not a minimal representative.

\begin{figure}[h]
\begin{center}
\includegraphics[scale=0.7]{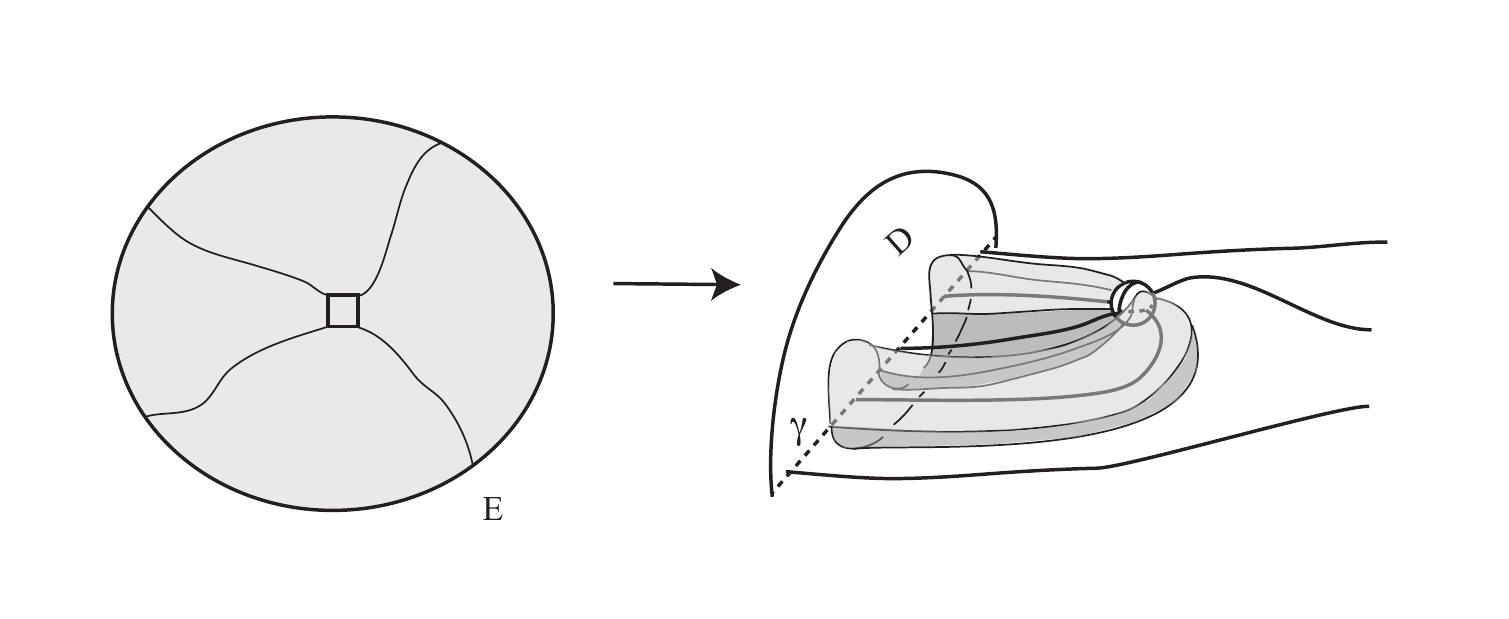}
\caption{A single saddle in the compression disk implies $\gamma$ is not a minimal representative.}
\label{newfork2}
\end{center}
\end{figure}

\end{proof}

\begin{proof}  (of Theorem \ref{main})  In order to prove that $S^3 - J'$ is hyperbolic, it is enough to show that $S^3-J'$ is irreducible, and to show there are no essential tori or annuli in $S^3-J'$. We begin with irreducibility, which is equivalent to showing that $S^3 - J'$ is not splittable. Since the alternating projection of $J$ is connected, Theorem  1(a) from \cite{Men} shows that $J$ is non-splittable. Hence, if $J'$ is splittable, there must be a sphere with $J$ to one side and at least one vertical component $C$  to the other side. Let $D$ be the vertical punctured disk bounded by $C$. Discarding the other vertical components, we work with just this one vertical component. But if $C$ is contained in a sphere, it bounds a disk in the sphere. We can use this disk to obtain a compression disk for $D$,  contradicting Lemma \ref{incom}. Hence, $J'$ is non-splittable. 

We now show that there are no essential tori in $S^3-J'$. Suppose $T$ is such a torus. Then $T$ is neither compressible nor boundary parallel. We first show that any such torus must be meridianally compressible. In other words, there is a nontrivial simple closed curve on the torus that bounds a disk punctured once by the link $J'$. 
 
 We again apply the techniques of \cite{Men}, which Menasco utilized to prove a similar result for alternating links. Suppose there is an essential torus $T$ that  is meridianally incompressible.   First,  we flatten each vertical component into the projection plane as in Figure \ref{flatten}. 

\begin{figure}[h]
\begin{center}
\includegraphics[scale=0.7]{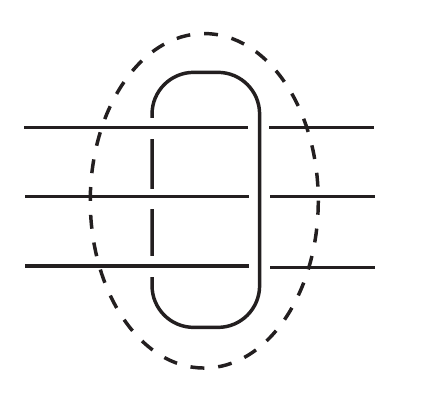}
\caption{Projecting the vertical components.}
\label{flatten}
\end{center}
\end{figure}

Since $T$ is incompressible and meridianally incompressible,  Lemma 1 of \cite{Men} tells us there exists a realization of $T$ such that the intersection curves with $S_+$ and $S_-$ do not intersect the same crossing bubble more than once. Moreover, every curve must intersect at least one bubble. We choose a realization for $T$ so that the number of saddles in bubbles corresponding only to $J$  is minimized.

We now eliminate the vertical components, without isotoping the surface $T$. Saddles that appeared in crossing bubbles involving the vertical components disappear and the intersection curves that entered a dotted region as in Figure \ref{flatten} now connect to one another. 

There are two fundamental changes in the system of intersection curves. First of all, each intersection curve need no longer bound a disk above $S_+$ in the case of $B_+$ and below $S_-$ in the case of $B_-$. Instead, a collection of intersection curves can bound a  subsurface of $T$ above $S_+$ or below $S_-$.  Second,  for each of the resulting intersection curves on $S_+$ and $S_-$, it can either be the case that the curve does or does not intersect a bubble more than once. 

We add the vertical components back in, but now they are once again vertical, perpendicular to the projection plane, each puncturing the projection plane in two points. For convenience, we consider intersection curves on $S_+$, but everything works just as well for intersection curves on $S_-$. We assume that $T$ was chosen to minimize the number of saddles.

The collection of intersection curves and saddles decompose $T$ into  squares on its surface corresponding to the saddles, and the components of intersections with $B_+$ and $B_-$.We first show that with the exception of  possibly a single $n$-punctured torus, all of these components are either disks or annuli. 

Let $R$ be a planar component of $T  \cap B_+$. Suppose an intersection curve $\alpha$ that forms one of the boundaries of $R$ on $S_+$ bounds a disk $F$ on $T$. We show that then $R$ is a disk.  Let $D'$ be the disk bounded by $\alpha$ on $S_+$. It may or may not contain additional intersection curves. If there are no vertical components with endpoints in $D'$, then we can isotope the disk $F$ bounded by $\alpha $ on $T$ to $D'$, pushing any other parts of $T$ out of the way in the process. After this isotopy, we have either eliminated $\alpha$ as an intersection curve, simplifying the intersection graph, or $R$ was a disk in $T \cap S_+$. If $D'$ is punctured  by some vertical component once, then that vertical component will be nontrivially linked with $\alpha$, contradicting the fact $\alpha$ bounds $F$. Hence, any vertical component $C$ that intersects $D'$ must do so with both of its endpoints. No other vertical components can be linked with this one above the projection plane, as in Figure \ref{linked}, since if they were, they would also have to be similarly linked below the projection plane, which the existence of $F$ prevents. 

\begin{figure}[h]
\begin{center}
\includegraphics[scale=0.3]{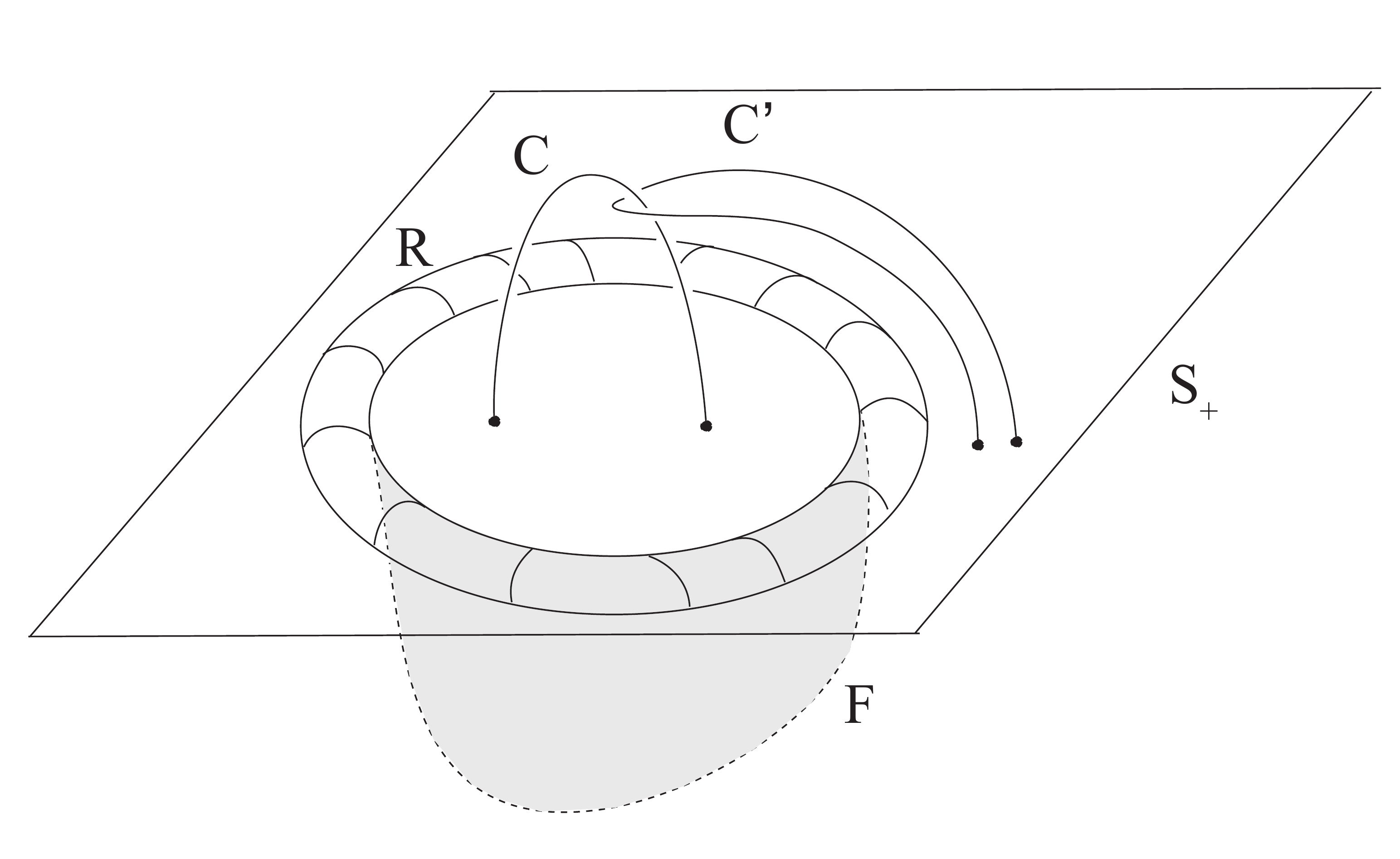}
\caption{Vertical components cannot be linked above the projection plane when an intersection curve bounds a disk on $T$.}
\label{linked}
\end{center}
\end{figure}

We can then take a disk $D''$ that is horizontal and has boundary in $R$ that is a curve parallel to $\alpha$ and slightly above it. Then we can isotope  $F$ to $D'$, eliminating the intersection curve $\alpha$ and any saddles that it touches, a contradiction to minimality. Note we are using the fact $J'$ is non-splittable here. So, the only planar components of $T \cap S_+$ are disks and annuli, with all boundary components of the annuli appearing as parallel nontrivial curves on the torus as in Figure \ref{torus}(b). 

\begin{figure}[h]
\begin{center}
\includegraphics[scale=0.7]{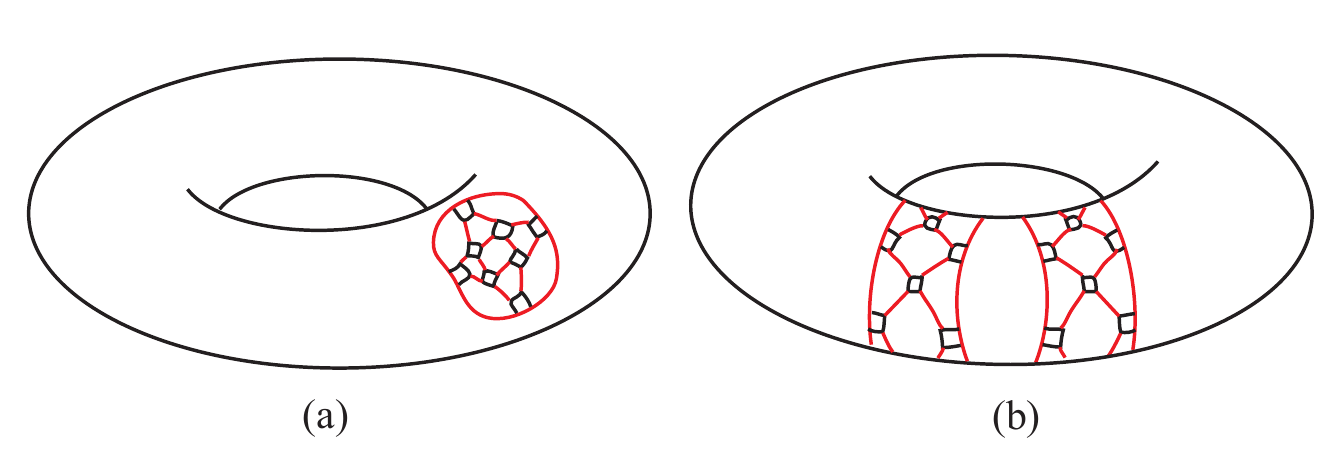}
\caption{Islands as in (a) can only occur on $T$ if there is a component in $T \cap B_{\pm}$ that is an $n$-punctured torus. Annular regions appear as in (b).}
\label{torus}
\end{center}
\end{figure}

We first show that there are no innermost curves on $S_+$ that bound disks in $B_+$.  A bubble that intersects an innermost curve is called an inner(outer) bubble if its overstrand lies to the inside(outside) of the curve. 

Such an innermost curve $\alpha$ bounding a disk $G$ in $T \cap B_+$ must intersect bubbles more than once, as $J$ is alternating, so the bubbles must alternate between bubbles with their overstrand to the right and bubbles with their overstrand to the left as we travel around the curve. Since there are no other curves inside $\alpha$, the other side of each inner bubble must be hit by $\alpha$ as in Figure \ref{typesII}. 

\begin{figure}[h]
\begin{center}
\includegraphics[scale=0.5]{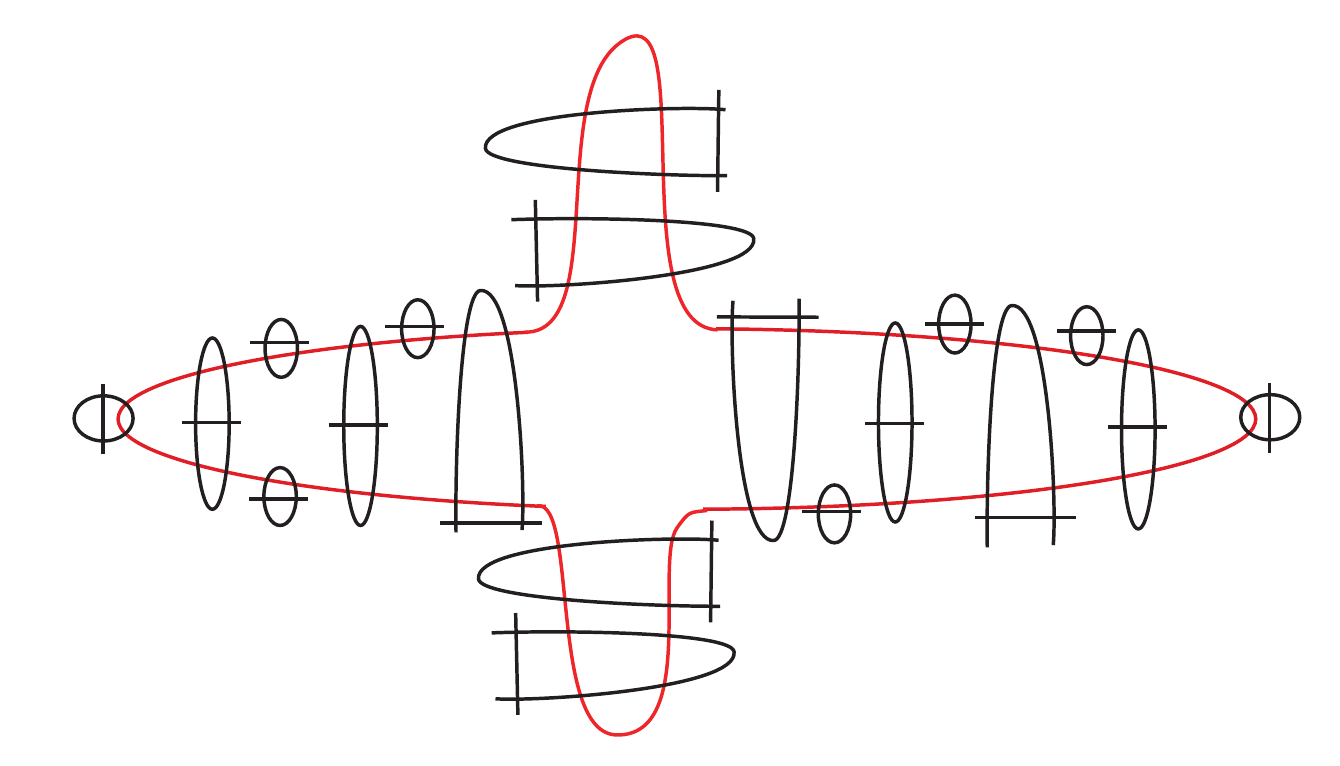}
\caption{Inner bubbles must be intersected on both sides by $\alpha$.}
\label{typesII}
\end{center}
\end{figure}

There are three types of bubble intersections with $\alpha$. A bubble of Type I intersects $\alpha$ on both sides of the bubble and the overstrand lies to the inside of $\alpha$. A Type II bubble is one that is intersected by $\alpha$ at least twice on one side and the overstrand lies to the outside of $\alpha$. A bubble of Type III intersects $\alpha$ once and has its overstrand to the outside of $\alpha$. If a bubble is of Type II or III, we consider only multiple intersections of that bubble occurring to the inside of the curve. We do not care if distinct bubbles appearing inside $\alpha$ are in fact the same bubble when also considered outside $\alpha$.

Choosing an innermost pair of intersections of $\alpha$ with a bubble of Type I, one on each side of the bubble, we we can form a loop $\gamma$ out of an arc on the corresponding saddle and an arc on $G$ that forms the boundary for a meridianal compression to the overstrand of the bubble. Since there are no meridianal compressions, it must be the case that one or more vertical components block this meridianal compression as in Figure \ref{block}(a).

\begin{figure}[h]
\begin{center}
\includegraphics[scale=0.7]{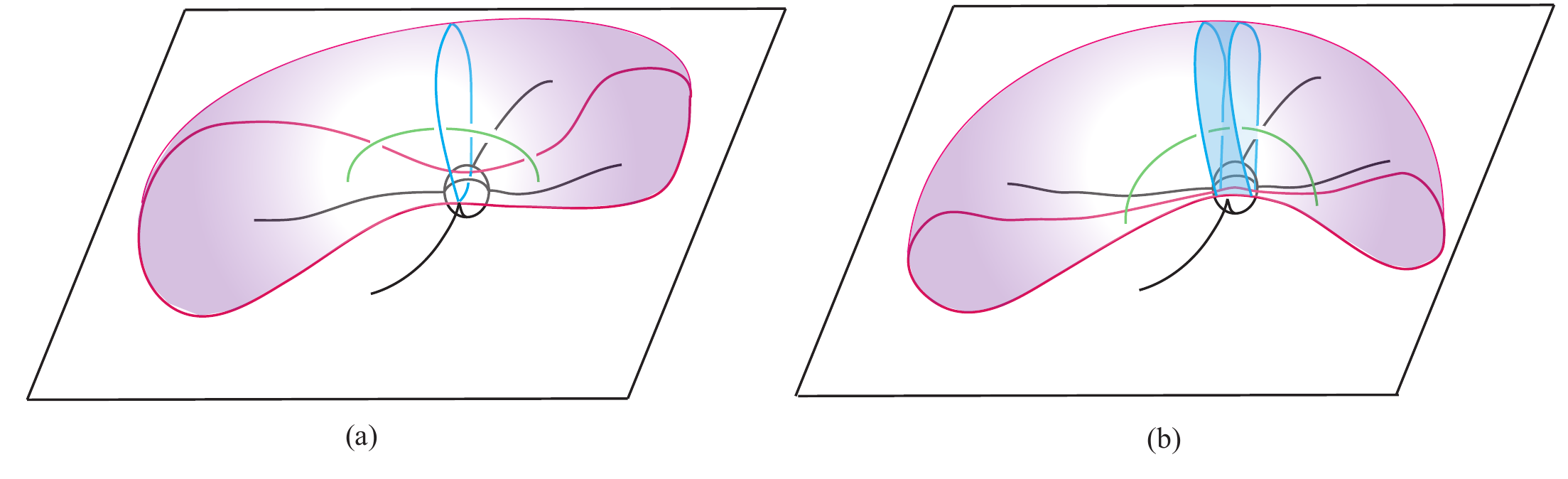}
\caption{A vertical component blocking meridianal compressions and saddle reducing isotopies.}
\label{block}
\end{center}
\end{figure}

Similarly, for a bubble of Type II, we can isotope a band on the disk bounded by $\alpha$ to eliminate two saddles unless the isotopy is blocked by one or more vertical components, as in Figure \ref{block}(b).

Relative to the vertical components, the intersection curve $\alpha$ can wind around the curves as in Figure \ref{messyintersectioncurve}, so when $\alpha$ is drawn uncomplicated, it could be the case that the vertical components are tangled with one another as if in a plat.

\begin{figure}[h]
\begin{center}
\includegraphics[scale=0.7]{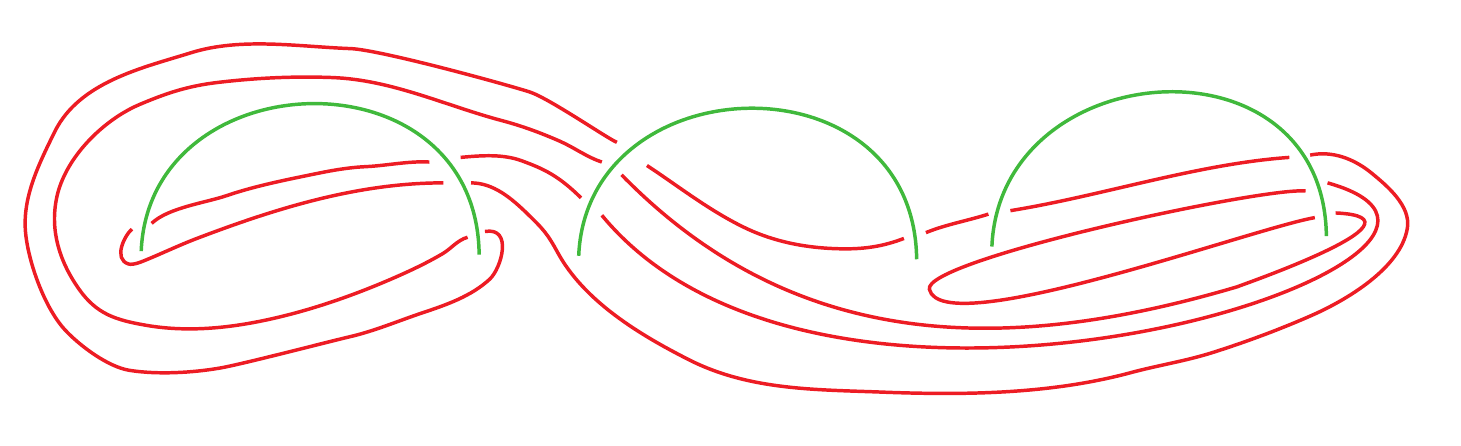}
\caption{The intersection curve $\alpha$ can wind around the vertical components..}
\label{messyintersectioncurve}
\end{center}
\end{figure}

We note the following {\it pairing property}. Consider all intersection curves on $S_+$  that intersect a vertical disk $D$ bounded by vertical component $C$, which is not necessarily in a minimal representation. In $B_+$, they must form the boundaries of surfaces in $T \cap B_+$. Then $T \cap B_+\cap D$ is a collection of arcs that pair the points in $\alpha \cap D$, some potentially nested. However, there must be an innermost arc such that it cuts a disk from $D \cap B_+$ that contains no other such arcs. Hence two adjacent intersection curves are connected by an innermost arc on $D$. We consider those two curves an {\it innermost pair}.

     We now utilize the pairing property. Suppose $C$ is a vertical component bounding a vertical punctured disk $D$ that blocks one or more meridianal compressions or saddle-reducing isotopies caused by bubbles to the inside of innermost curve $\alpha$ bounding the disk $G$ in $T \cap S_+$.



Choose any vertical component $C$ with endpoints in $\alpha$ that blocks crossings from generating either a meridianal compression or an isotopy lowering the number of saddles. Notice that regions $H$ and $M$ must have a vertical component that has exactly one endpoint in the region. Isotope $C$ so that the arc that represents its projection follows $\alpha$ as closely as possible, as in Figure \ref{typesIInewer}, and let $D$ be the punctured disk that it bounds. Note that we may have to carry along other vertical components with which it is entangled.

\begin{figure}[h]
\begin{center}
\includegraphics[scale=0.7]{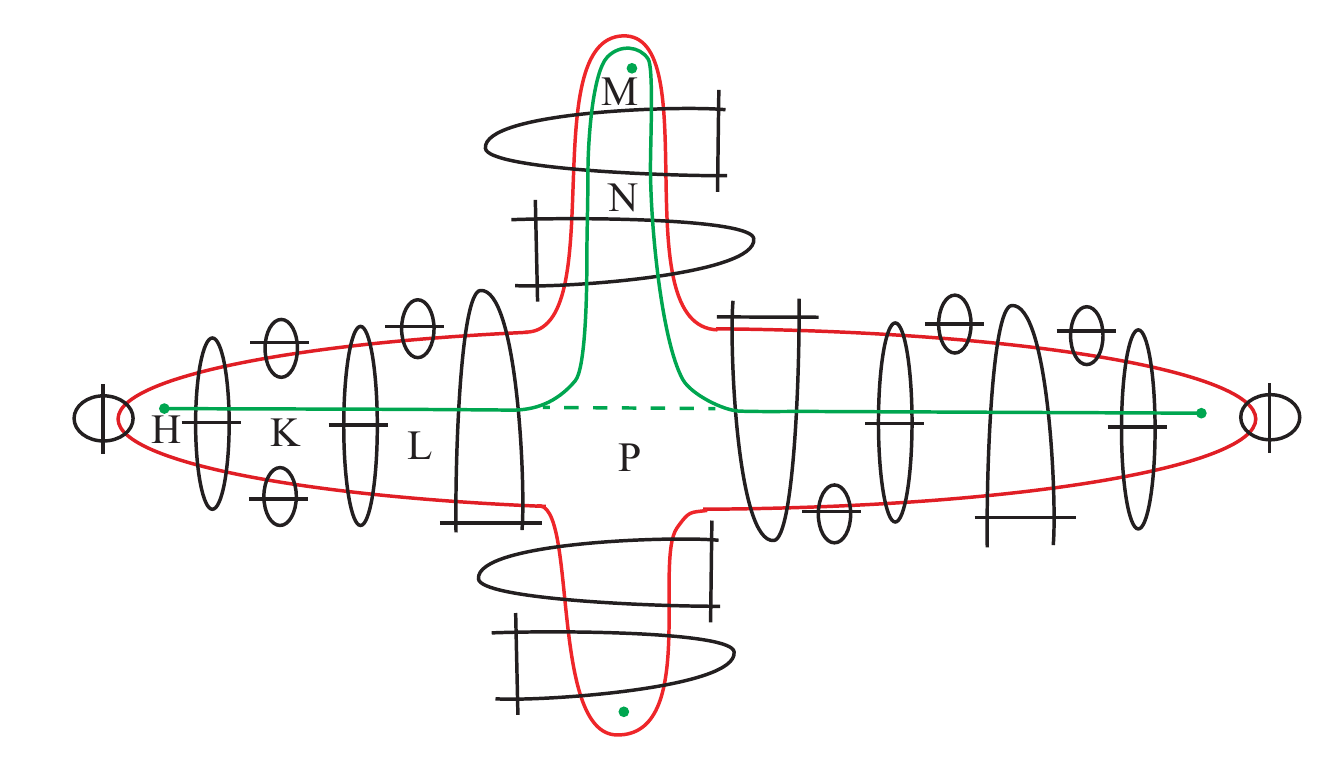}
\caption{A vertical component blocking a meridianal compression.}
\label{typesIInewer}
\end{center}
\end{figure}

Regions between bubbles inside $\alpha$ fall into six types, denoted H, K, L, M, N and P in Figure \ref{typesIInewer}..

     Shifting from the intersection curves on $S_+$ to the corresponding intersection curves on $S_-$. we find that each pair of adjacent intersection curves passing through $D$ share a bubble. See Figure \ref{typesIInewerA}. Note that the only region that might potentially cause a problem is Type P, which is we call a  juncture. If we had left our arc representing the vertical component to go straight across a juncture, there would have been a pair of adjacent curves passing under the vertical component that would not have shared a bubble. This is why we must instead follow $\alpha$ around the outside.
     
     \begin{figure}[h]
\begin{center}
\includegraphics[scale=0.7]{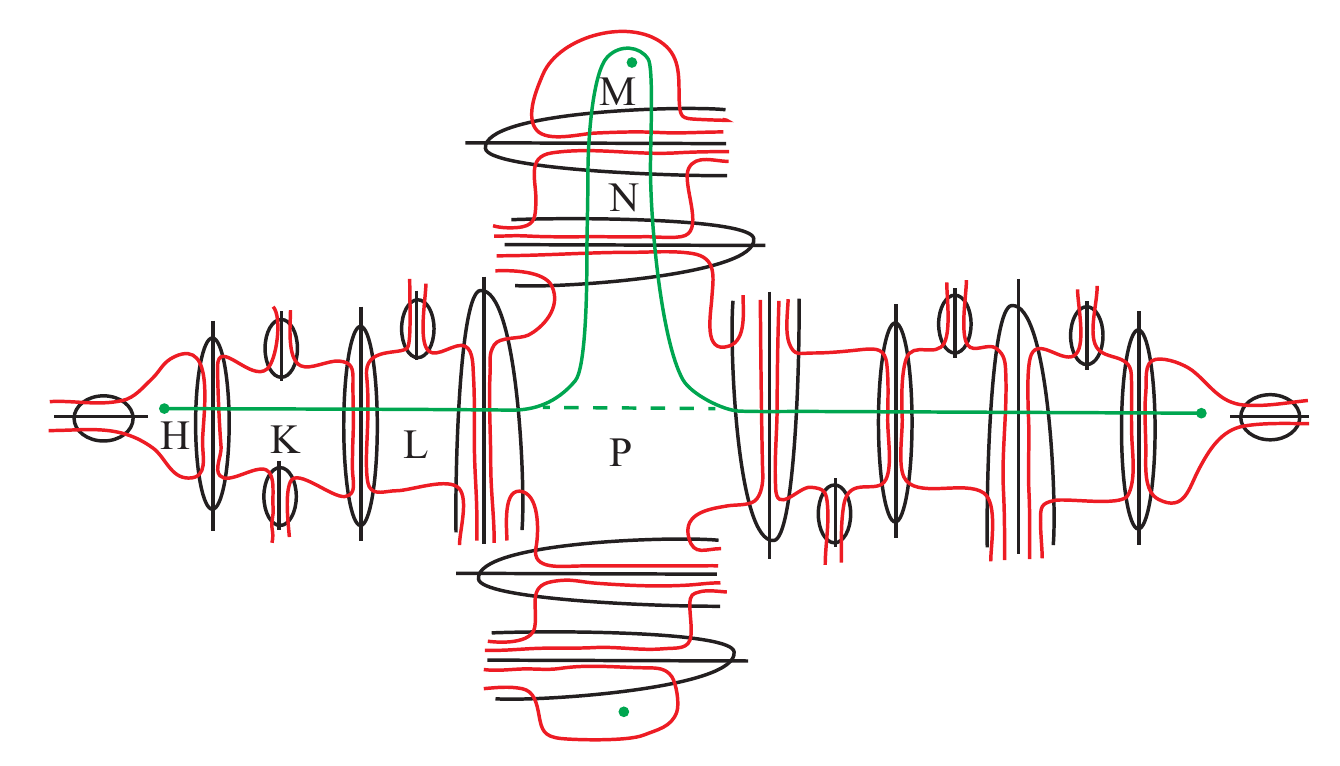}
\caption{The view from $S_-$, where all adjacent curves passing under the vertical component share a bubble.}
\label{typesIInewerA}
\end{center}
\end{figure}
     
    So there exists an innermost pair of adjacent curves passing under $C$. They must share a bubble. That bubble cannot be blocked by a vertical component because the disk $G$ prevents any such vertical component. Thus, we obtain either a meridianal compression if the bubble is a Type I or Type III bubble,  or a saddle reducing isotopy if the bubble is a Type II bubble. Hence, there can be no innermost intersection curve bounding a disk in $B_+$.
     
\begin{figure}[h]
\begin{center}
\includegraphics[scale=0.7]{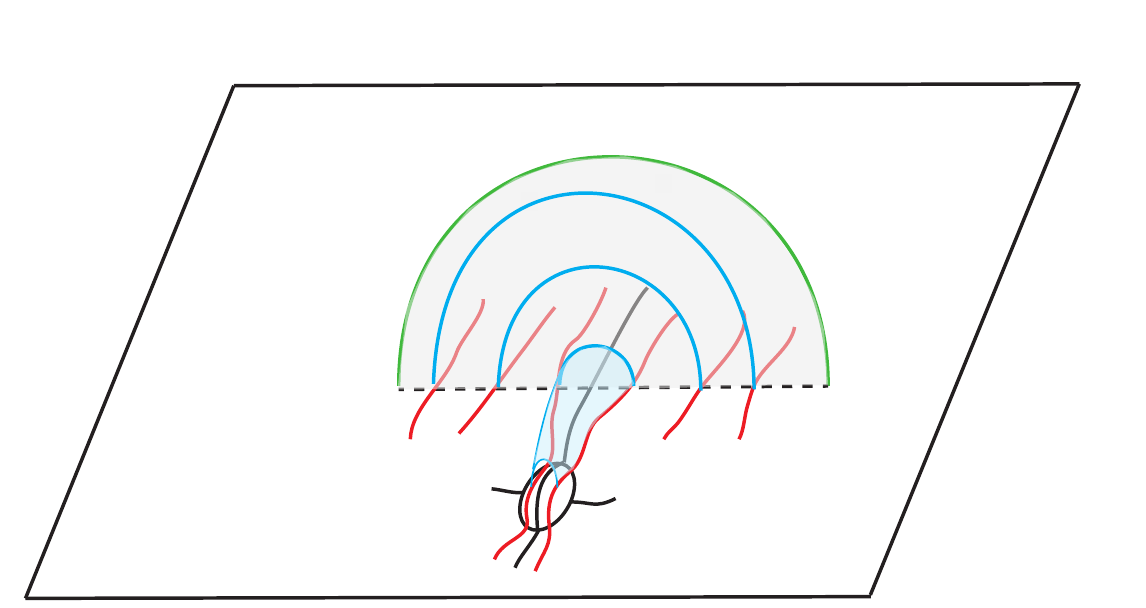}
\caption{Two intersection curves share both a bubble and a surface, yielding a meridianal compression.}
\label{block3}
\end{center}
\end{figure}

We now prove there are no annular components in $T \cap B_+$. Suppose there were such an  annulus  $A$. Note that the existence of $A$ precludes the possibility of a punctured torus component in either $T \cap B_+$ or $T \cap B_-$. Each of its boundary components bounds a disjoint disk in $S_+$ and those disks, together with $A$ bound a ball in $B_+$. Choose an annulus that is innermost in the sense that its ball contains no other ball corresponding to such an annulus. Note that $T$ then bounds a solid torus $V$ to the ball side of $A$. Because $T$ is incompressible and meridianally incompressible, there must be a set $\mathcal{C}$ of at least two vertical components that together prevent $A$ from being compressible or meridianally compressible in $V$. One possibility is that one or more of these components intersect the projection plane in the distinct disks bounded by $\partial A$. But it could also be the case that two or more vertical components are linked together above the projection plane inside the ball bounded by $A$. In this case, the same vertical components must be similarly linked  beneath the projection plane, since ultimately these vertical components form an unlink when considered as a whole. In either case, the solid torus $V$ must be unknotted, as $\mathcal{C}$ intersects every meridianal disk in $V$ and if $V$ were knotted, this would make $\mathcal{C}$ a nontrivial link by itself, when all other components are  dropped, which is a contradiction to how we constructed $J'$.

Again, in either case, there is a compression disk for $A$ that intersects only vertical components. Dropping all vertical components momentarily, that compression yields a sphere and since $J$ is non-splittable, it must lie to one or the other side of that sphere. Hence, $J$ lies to one or the other side of $T$. If $J$ is in $V$, then to avoid compressions and meridianal compressions, there must be vertical components to the other side. However, then again, we find that by dropping $J$ temporarily, the collection of vertical components forms a nontrivial link, a contradiction.

Hence, it must be the case that $J$ lies to the outside of $V$. Let $A'$ be any other annulus in $T \cap B_+$. Then its ball must also intersect $\mathcal{C}$ since $\mathcal{C}$ is the only collection of vertical components that together wrap all the way around $V$. Hence the curves $\partial A'$ must be parallel to the curves in $\partial A$. This implies that there are no disk components in $T \cap B_+$, as if there were such, there would need to be an annular component to each side of its boundary on $S_+$  to not be innermost, which the existence of $\mathcal{C}$ prevents.

Hence, we have only annular components remaining in either $T \cap B_+$ or $T \cap B_-$. However, as in Figure \ref{torus}(b), if any intersection curve bounding an annulus  intersects a bubble, there must be disk components. Hence, all intersection curves avoid bubbles. But then the boundaries of the annuli separate the projection of $J$. We could isotope any annuli away that are not parallel to $A$, so it must be the case that all annuli are parallel to $A$. Similarly in $B_-$, all annuli must be parallel. Hence, there can be only one annulus to either side for $T$ to be connected. Since $J$ is to the outside of $V$, each disjoint disk on $S_+$ bounded by $\partial A$ lies in a single complementary region of $J$. But since the endpoints of the vertical components  in $\mathcal{C}$ lie in these two disks, each such component either has both of its endpoints in the same complementary region or its two endpoints share the same complementary regions as another vertical component. In either case, this contradicts our construction of $J'$.

The last case to consider is when there is a component of $T \cap B_+$ that is an $n$-punctured torus. In this case, all other components of $T \cap B_+$ and $T\cap B_-$ must be disks. However, then all components to the $B_-$ side are disks, a possibility we have eliminated.

Thus, we have shown that an essential torus $T$ is meridianally compressible. If such a meridianally compressible torus exists,  we can meridianally compress it to obtain a twice-punctured sphere. Since $T$ was not boundary-parallel, this twice-punctured sphere implies that the link $J'$ is composite.

We  now show that the link $J'$ is prime. That is to say, we show that there are no essential annuli with both boundaries appearing as meridianal curves for link components. 

We first suppose that the boundary components of $A$ are meridians on $\partial N(J)$. So we think of $A$ as a twice-punctured sphere. Then $A$ demonstrates that $J'$ is a composite link. Since $J'$ is not composite,  it must be that the addition of the vertical components  prevents $A$ from being boundary-parallel. We assume that $A$ is not meridianally compressible by doing any compressions first and taking only one of the resultant annuli to consider. As we did with $T$, we use the results of \cite{Men} to put the punctured sphere in standard position relative to $S_+$ and $S_-$ so that no intersection curve intersects a bubble more than once and every curve either has a puncture or a bubble on it. We assume that $A$ has been chosen to minimize the number of saddles in crossing bubbles only involving $J$. We again throw away the vertical components and obtain intersection curves that no longer need to bound disks in $S_+$ and $S_-$,  and that can intersect a given bubble more than once. As in \cite{Men}, since $J$ is alternating,  a curve that crosses a bubble with its overstrand to one side of the curve must pass through an odd number of punctures (there are only two total) if it subsequently passes through another bubble with its overstrand to the same side.

The same proof we used in the case of an essential meridianally incompressible torus shows that no components of $A \cap B_{\pm}$ can be other than disks and also annuli with nontrivial boundaries on $A$. In fact,  no component of $A \cap B_{\pm}$ can be such an annulus, as any compression disk for such an annulus would  intersect only vertical components, but every compression disk for an essential annulus in $A$ must intersect $J$ since $J$ punctures the sphere corresponding to $A$ twice.

We could have up to four disks in $A \cap B_{\pm}$ with punctures on their boundaries, which, in addition to the types of regions depicted in Figure \ref{typesIInewer}, also allow for regions Q, R and S depicted in Figure \ref{typesIInewest}. But again, the argument given previously applies to show that there are no disks innermost on $S_{\pm}$ in $A \cap B_{\pm}$ that intersect bubbles, either with or without punctures on their boundary.

\begin{figure}[h]
\begin{center}
\includegraphics[scale=0.7]{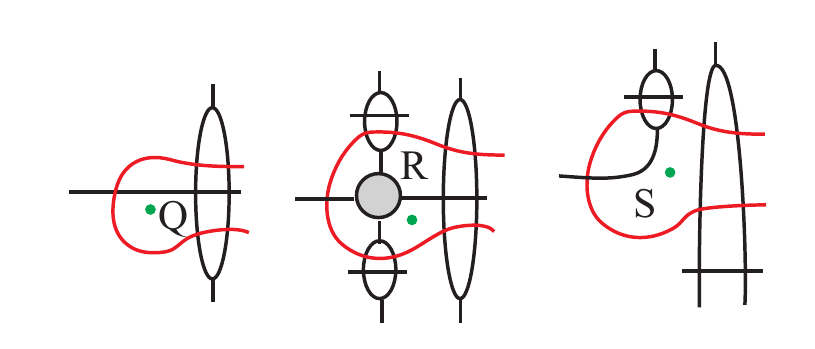}
\caption{Additional possible regions for an innermost disk with a puncture on its boundary.}
\label{typesIInewest}
\end{center}
\end{figure}


But then $A$ intersects $S_+$ and $S_-$ in the same pair of arcs from the first puncture to the second puncture, each of which does not intersect the projection of $J$. So, to one side, the projection of $J$ is a trivial arc. There can be no vertical components to this side of $A$ since there are only two adjacent regions of the projection plane to this side. This shows that $A$ cannot be an essential annulus with both boundaries meridianal on $N(J)$.

We now consider an essential annulus $A$ with both boundary components meridians on boundaries of neighborhoods of the vertical components. It is appropriate to consider  $A$ as a twice-punctured sphere. But then both punctures must be from one vertical component $C$  and $J$ must lie to one side. To the other side is a trivial arc of $C$ and additional vertical components that prevent $A$ from being boundary-parallel. But individual vertical components inside this twice-punctured sphere bound disks there contradicting the fact that the $n$-punctured disks they bound are incompressible. 

Note that there can be no annulus with one boundary a meridian on $J$ and another boundary a meridian on a vertical component, as a sphere in $S^3$ cannot be punctured once by a simple closed curve. Thus, there can be no essential annuli with meridianal boundary components in $S^3-J'$ and $J'$ is prime. Since we have shown that any essential torus must be meridianally compressible, which would yield just such an essential annulus, there are no essential tori in $S^3-J'$.

We now consider the possibility  of other essential annuli. Lemma 1.16 of \cite{Hatcher} implies that if there is an essential annulus at all, then $S^3 - J'$ is Seifert fibered, with the boundaries of the annulus as fibers. Moreover, there are either a total of one, two or three torus boundary components in $S^3 - N(J')$. If one boundary of $A$ is a meridian on a vertical component $C$, and the other is not,  then if we fill $C$ in, the annulus becomes a compressing disk on the boundary of the resulting link. If the other boundary component lies on the boundary of a neighborhood of $J$, this contradicts the hyperbolicity of $J$. If the other boundary is on a boundary of a neighborhood of a different vertical component, we have that a nontrivial curve on its boundary is trivial in the complement of $J$. However, this can never occur for a trivial link component unless the curve is a longitude, but then we are contradicting the incompressibility of the vertical disk bounded by the component. 

If a boundary-component of $A$ is a non-meridian on a vertical component $C$, then we can fill in $C$ and extend the Seifert fibration to the solid torus that we filled in, making the result Seifert fibered, which is a contradiction to the hyperbolicty of $S^3-J$ unless there is a second vertical component. If there is a second vertical component and the other boundary of $A$ is a non-meridian upon it (as it must be by our previous considerations), then we can fill it in also, and we obtain a Seifert fibration for $S^3-J$, a contradiction to its hyperbolicity. 

The only possibility left is that the boundaries of $A$ is a non-meridian upon both $C_1$ and a component $K$ of $J$, which is all of $J$, and there is a second vertical component $C_2$. Then $S^3-J'$ is a twice-punctured disk crossed with a circle obtained by taking  a regular neighborhood of $A \cup\partial N(K)$. This is embedded in $S^3$ so that each boundary torus bounds a solid torus to the exterior, which is a neighborhood of the corresponding link component. But then the component corresponding to the outer boundary of the disk, which is one of the vertical components, links both of the other components, one of which is also vertical. However, then two vertical components are linked, a contradiction to the construction of $J'$.
 
We now consider an essential annulus $A$ with  both boundary components on $N(J)$ but not meridians. Then they must both be on the same component or else $A$ would be essential in the complement of $J$, a contradiction. In this case,  each boundary of $A$ is a $(p.q)$-curve on the boundary of a neighborhood of the link component $K$, with 
$ |q| \geq 1$. Hence,  $C$ is a $(p,q)$-cable of $K$. But then  there is an essential annulus with one boundary on $\partial N(C)$ and a second boundary on $N(K)$, a possibility we have already eliminated.

\end{proof}

\section{Generalized Belted Sums}

In this section, we show that if $L_1 \#_bL_2$ is constructed from two links with a $2n$-string tangle at center, with $n$ belts around adjacent pairs of the exiting strings and $n-3$ belts around the central tangle, no two of which are parallel, as for instance appears in Figure \ref{beltedsum2} in the case $n=5$, then $\mbox{vol}(L_1 \#_bL_2) =\mbox{vol}(L_1) + \mbox{vol}(L_2) - 4(n-2)3.6638...$. To see this fact, we utilize the thrice-punctured spheres that appear in the link complement. Thrice-punctured spheres are known to be totally geodesic with a rigid structure  in a hyperbolic 3-manifold (see for instance \cite {Adams2}). In particular, any two are isometric. In the case of such a $2n$-string tangle $T_1$, there is a collection of thrice-punctured spheres that shield the part of the manifold corresponding to the $2n$-string tangle from the rest of the manifold. Cutting the manifold open along this collection of thrice-punctured spheres and then for each resulting piece,  doubling across the thrice-punctured spheres yields two link complements, one with the $2n$-string tangle to the inside and outside and the other of which is an untwisted daisy chain with additional components, as appear in Figure \ref{links}. The two halves of the original manifold must have volume exactly half of these, since the reflection doubles the volume. So the original manifold has volume exactly half the sum of these two volumes. The same is true for the link with $2n$-string tangle $T_2$. Now when we take the  two link complements, cut them both open along the collection of thrice-punctured spheres, throw away the two halves of the untwisted daisy chain, we obtain the volume of the first link plus the volume of the  second link minus the volume of the untwisted daisy chain with additional components.

In the case $n=3$, the volume of the untwisted daisy chain is 4(3.6638...) where 3.6638... is the volume of an ideal regular octahedron. The manifold is commensurable with the Whitehead link. (See Example 6.8.7 of \cite{Thurston}.) For $n>3$, we can cut the link complement open along the twice-punctured disks bounded by components that are not in the untwisted daisy chain to obtain $n-2$ pieces, each of volume $\frac 12 (4(3.6638...))$ as in Figure \ref{links}. When we take the belted sum of the two links, we discard all of these pieces from both link complements, meaning we lose a volume of $4(n-2) 3.6638...$.

\begin{figure}[h]
\begin{center}
\includegraphics[scale=0.7]{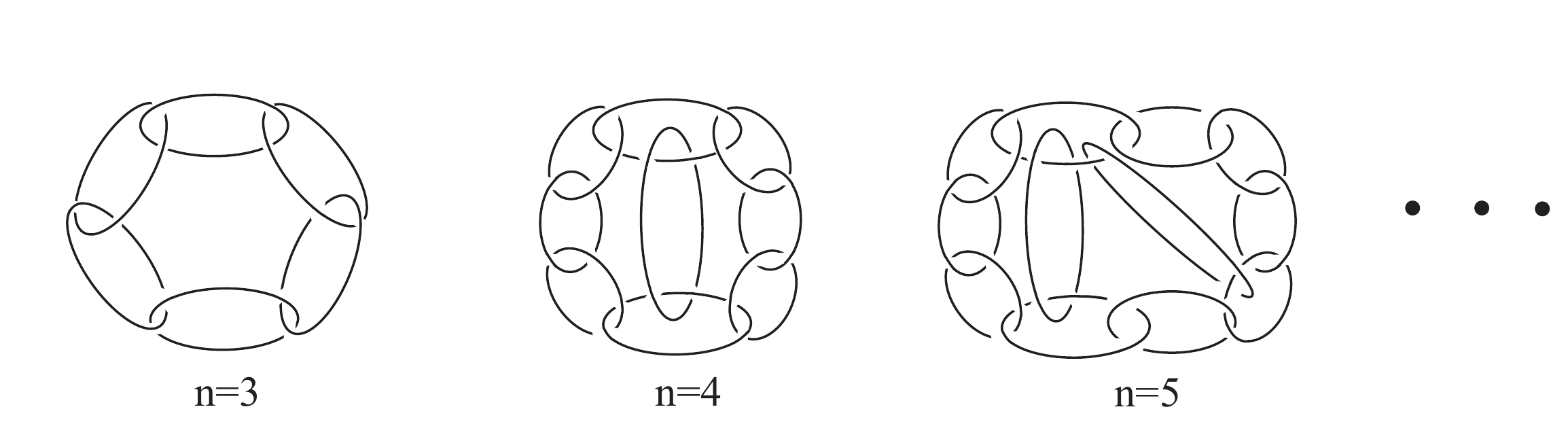}
\caption{Links of volume $4(n-3) 3.6638...$}
\label{links}
\end{center}
\end{figure}

We can further start with any link  and add components to decompose it into pieces, each of which has a volume we can determine as in the case of a generalized belted sum, to obtain the volume of the augmented link, which will bound the volume of the original link since Dehn filling always decreases volume. In the case of an alternating link, Theorem \ref{main} tells us that the generalized augmented link that we produce will be hyperbolic, which we need to know for the procedure to apply.  As an example, consider the link appearing in Figure \ref{bigexample}. We denote the link obtained from a $2n$-tangle $T_i$ by completing it as in Figure \ref{beltedsum2} by $L_i$.

\begin{figure}[h]
\begin{center}
\includegraphics[scale=0.4]{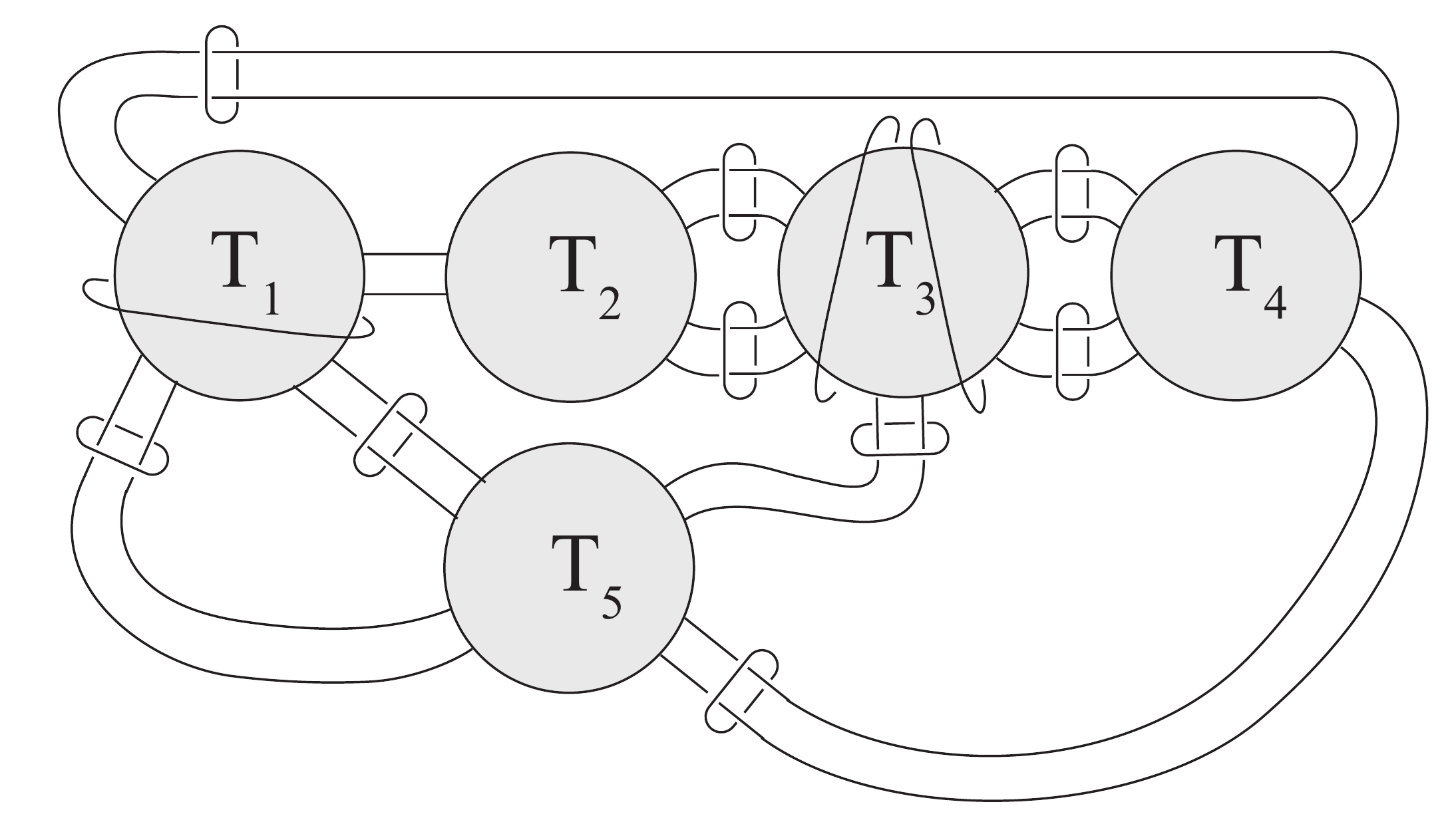}
\caption{Finding the volume of this link.}
\label{bigexample}
\end{center}
\end{figure}

We can cut along the various thrice-punctured spheres, and realize each of the resulting pieces as a link of the appropriate type, where we have thrown away a volume (n-2)3.6638... . In this case, we decompose the link complement into three 8-tangles, one 6-tangle and one 10-tangle.
 There is also another piece remaining which is a copy of the Borromean rings. Hence the volume is \
 $\mbox{vol}(L_1) + \mbox{vol}(L_2) + \mbox{vol}(L_3) + \mbox{vol}(L_4) + \mbox{vol}(L_5)- 20(3.6638...) + 7.32772...$.

\end{document}